\documentclass[11pt,a4paper]{article}
\usepackage[left=2cm, top=3cm,bottom=3cm,right=2cm]{geometry}
\usepackage{mathtools,amssymb,amsthm,mathrsfs,calc,graphicx,xcolor,cleveref,dsfont,tikz,pgfplots,bm,url,tabularx}
\usepackage[british]{babel}
\usepackage{amsfonts}              
\usepackage[T1]{fontenc}
\usepackage{enumitem}
\usepackage[color=green]{todonotes}
\usepackage{tikz-cd}
		\usetikzlibrary{calc} 
		\usetikzlibrary{intersections}
\usepackage[labelfont=bf]{caption}
\usepackage[font=small]{subcaption}

\theoremstyle{plain}
\newtheorem{theorem}{Theorem}

\newtheorem{corollary}[theorem]{Corollary}

\theoremstyle{definition}

\newtheorem{example}[theorem]{Example}

\theoremstyle{remark}
\newtheorem{remark}[theorem]{Remark}

\newcommand{\dint}{\textup{d}}

\newcommand{\vol}{\mathop{\mathrm{vol}}\nolimits}

\def\CC{\mathbb{C}}

\def\EE{\mathbb{E}}

\def\HH{\mathbb{H}}

\def\OO{\mathbb{O}}
\def\PP{\mathbb{P}}

\def\RR{\mathbb{R}}

\makeatletter
\let\@fnsymbol\@alph
\makeatother

\begin{document}

\title{\bfseries Visibility in the Boolean Model \\ on Harmonic Manifolds}

\author{Enkelejd Hashorva\footnotemark[1]\,\, and Christoph Th\"ale\footnotemark[2]}

\date{}
\renewcommand{\thefootnote}{\fnsymbol{footnote}}
\footnotetext[1]{University of Lausanne, Switzerland. Email: enkelejd.hashorva@unil.ch}
\footnotetext[2]{Ruhr University Bochum, Germany. Email: christoph.thaele@rub.de}

\maketitle

\begin{abstract}
	\noindent
	In Poisson Boolean models with deterministic ball grains, the directional visible range from an uncovered point is known to be exponentially distributed in Euclidean and real hyperbolic space. We show that the same phenomenon holds on every simply connected non-compact homogeneous harmonic manifold. The geometric mechanism behind this fact is the affine-linear growth of tube volumes around geodesic segments. As a consequence, we identify the finiteness regime for the expected volume of the visible region, including a geometric interpretation of the critical threshold in the positive-entropy case. We also construct explicit complete non-homogeneous Riemannian manifolds showing that exact exponentiality is tied to exact tube linearity: superlinear tube growth leads to Weibull-type tails, while asymptotic tube linearity still yields an exponential decay rate.
	\\
	
\end{abstract}


\section{Introduction and summary of the results}

A classical question in stochastic geometry asks how far one can see in a prescribed direction through a random medium before the line of sight is blocked. In its simplest form, this problem is naturally modelled by the Boolean model. In the standard Euclidean setting, the Boolean model is obtained by placing grains, for example balls, at the points of a homogeneous Poisson process and taking the union of all grains. In the present paper, we focus on the case of deterministic ball grains, since this already captures the geometric mechanism underlying the phenomenon we wish to study.

Let \(M\) be a simply connected non-compact homogeneous Riemannian manifold. We denote by \(\vol_M\) the Riemannian volume measure on \(M\). Let \(\eta\) be a Poisson point process on \(M\) with intensity measure \(\lambda\vol_M\), where \(\lambda>0\), and fix \(\rho>0\). The associated Boolean model is the random closed set
\[
Z:=\bigcup_{x\in\eta} B(x,\rho),
\]
where \(B(x,\rho)\) denotes the closed geodesic ball of radius \(\rho\) centred at \(x\). We are interested in the geometry of the vacant region \(M\setminus Z\), and in particular in the visible range from a point that is not covered by the occupied phase $Z$.

To make this precise, fix a reference point \(o\in M\), condition on the event \(o\notin Z\), and let \(u\) be a unit tangent vector at \(o\). Writing \(\gamma_u:[0,\infty)\to M\) for the geodesic ray starting at \(o\) in direction \(u\), we define the visible range in direction \(u\) by
\[
s(u):=\inf\{r\geq 0:\gamma_u(r)\in Z\},
\]
with the convention that $\inf\varnothing=\infty$. Equivalently,
\[
s(u)=\sup\{r\geq 0:\gamma_u([0,r])\subset M\setminus Z\}.
\]
Thus, \(s(u)\) is the maximal distance one can travel along the ray \(\gamma_u\) before first hitting the occupied phase.

In Euclidean space, a direct application of the Poisson avoidance formula shows that \(s(u)\), conditioned on \(o\notin Z\), has an exponential distribution. Remarkably, the same phenomenon persists in real hyperbolic space as shown in \cite{BuehlerHugThaeleBM}. At first sight, this may seem surprising: Euclidean space is flat and has polynomial volume growth, whereas hyperbolic space has constant negative curvature and exponential volume growth. The fact that directional visibility exhibits the same distributional form in these geometrically very different settings naturally leads to the following question:
\medskip

\noindent
\textit{Which geometric properties of the underlying manifold ensure that the
	visible range in a fixed direction has an exponential distribution?}

\medskip

The purpose of the present paper is to investigate this question. Our main result identifies a rigid geometric mechanism behind the exponential law: it is governed by the volume growth of tubes around geodesic segments. More precisely, we show that the exponential visibility property holds on simply connected non-compact homogeneous harmonic manifolds, thereby recovering both the Euclidean and the real hyperbolic cases in a unified framework. At the same time, this yields a new and simplified proof of \cite[Proposition 6.1]{BuehlerHugThaeleBM} in the special case of deterministic ball grains.

Recall that a connected Riemannian manifold is called \emph{harmonic} if the volume density in normal coordinates about any point depends only on the radial variable. Equivalently, small geodesic spheres have mean curvature depending only on the radius, see \cite{GrayBook,KnieperSurvey}. The homogeneous harmonic manifolds are particularly rigid. By a theorem of Heber \cite{HeberHarmonicHomogeneous}, every simply connected homogeneous harmonic manifold is either flat, a rank-one symmetric space of non-compact type, or a non-symmetric Damek--Ricci space, for which we refer to \cite{DRSpaces}. Thus, the class considered here contains the Euclidean spaces, the real, complex and quaternionic hyperbolic spaces, the Cayley hyperbolic plane, and the non-symmetric Damek--Ricci spaces.

The relevant geometric input is the tube property: For every fixed radius \(\rho>0\), the volume of the \(\rho\)-tube around a geodesic segment of length \(r\) is affine linear in \(r\). More precisely, there exists a constant \(a_\rho\in(0,\infty)\), depending only on \(M\) and \(\rho\), such that
\[
\vol_M(T_\rho(\gamma_u[0,r]))
=
\vol_M(B(o,\rho))+a_\rho r,
\qquad r\geq 0.
\]
Here, \(T_\rho(\gamma_u[0,r])\) denotes the closed \(\rho\)-neighbourhood of the geodesic segment \(\gamma_u([0,r])\).
Combining this with the Poisson avoidance formula gives
\[
\PP\bigl(s(u)>r\,\big|\,o\notin Z\bigr)
=
\exp(-\lambda a_\rho r),
\qquad r\geq 0.
\]
Thus, \(s(u)\), conditionally on \(o\notin Z\), is exponentially distributed with parameter \(\lambda a_\rho\).

As an application of this result, we investigate the mean volume of the entire visible region from the reference point. Conditioned on \(o\notin Z\), let \(V_o\) be the set of points \(x\in M\) for which the (unique) geodesic segment from \(o\) to \(x\) does not intersect \(Z\). In real hyperbolic space it was shown in \cite{BuehlerHugThaeleBM} that the mean visible volume undergoes a phase transition. Namely, it is finite if and only if the intensity \(\lambda\) is strictly above a critical value depending on the dimension and on the radius \(\rho\). We prove the corresponding statement in the present generality. For this, let
\[
h:=\lim_{r\to\infty}{\log\vol_M(B(o,r))\over r}
\]
be the volume entropy of \(M\). In the homogeneous harmonic examples considered here, the geodesic sphere volumes satisfy a two-sided exponential bound with exponent \(h\), if $h>0$. Hence, using geodesic polar coordinates and the exponential law for \(s(u)\), the finiteness of the mean visible volume is governed by the comparison between \(\lambda a_\rho\) and \(h\). More precisely, we show that
\[
\EE[\vol_M(V_o)\mid o\notin Z]<\infty
\qquad\Longleftrightarrow\qquad
\lambda > {h\over a_\rho}.
\]
This also includes the flat case \(h=0\): in Euclidean spaces the expected visible volume is finite for every \(\lambda>0\). If \(h>0\), this gives the critical intensity $\lambda_c(M,\rho)={h\over a_\rho}$ above which the conditional mean visible volume is finite. At the same time, the explicit formula for $\lambda_c(M,\rho)$ delivers a geometric interpretation: it is the ratio between the exponential volume growth rate of the ambient space and the linear growth rate of tubes around geodesic segments.

Finally, we show that the homogeneous harmonic assumption cannot simply be discarded. We construct a complete non-homogeneous Riemannian surface on which the tube volume around a distinguished geodesic ray grows superlinearly. For the associated Boolean model, the conditional visible range in that direction is no longer exponentially distributed. Instead, its tail has a Weibull-type decay. This demonstrates that the exponential law is not merely a consequence of the Poissonian nature of the model, but is tied to the differential geometry of the underlying space. In this context, we also discuss an intermediate phenomenon. Exact affine linearity of tube volumes yields an exact exponential distribution for the directional visible range. If the tube volume is only asymptotically linear, then one should not expect an exact exponential law in general, but the visible range still has an exponential decay rate. We illustrate this by two explicit non-homogeneous examples in which an averaged transverse tube density exists along the chosen geodesic.

In \cite{LyonsRandomShadows}, Lyons has studied a related visibility problem on Cartan--Hadamard manifolds, that is, complete simply connected Riemannian manifolds with non-positive sectional curvature. The simply connected homogeneous harmonic manifolds considered in our main result form a special
subclass of this setting. The emphasis in the present paper is different, however. Lyons studies visibility to infinity, or equivalently coverage of the ideal boundary by random shadows, whereas we first compute the visible range in one fixed direction and then derive consequences for the expected visible volume. Our examples beyond the harmonic setting show that, even within the class of Cartan--Hadamard manifolds, exact exponentiality of the directional visible range is tied to exact affine-linear tube growth.

\section{A positive result for harmonic manifolds}

\subsection{Exponential visibility range from the tube property}

Throughout this section, \(M\) is a simply connected non-compact homogeneous
Riemannian manifold. We denote by \(\vol_M\) the Riemannian volume measure and by $d_g(\,\cdot\,,\,\cdot\,)$ the Riemannian distance on
\(M\). Let \(\eta\) be a Poisson point process on \(M\) with intensity measure
\(\lambda\vol_M\), where \(\lambda>0\), and fix \(\rho>0\). We consider the
Boolean model
\[
Z:=\bigcup_{x\in\eta} B(x,\rho),
\]
where \(B(x,\rho)\) denotes the closed geodesic ball of radius \(\rho\) centred
at \(x\).

Fix \(o\in M\), a unit tangent vector \(u\in T_oM\), and let
\(\gamma_u:[0,\infty)\to M\) be the corresponding geodesic ray with
\(\gamma_u(0)=o\) and \(\dot\gamma_u(0)=u\). For \(r\geq 0\), the event
\(\{s(u)>r\}\), where we recall that $s(u)$ denotes the visible range in the direction of $u$, is equivalent to the statement that the segment
\(\gamma_u([0,r])\) does not meet the occupied phase $Z$, that is,
\[
\{s(u)>r\}
=
\bigl\{\gamma_u([0,r])\cap Z=\varnothing\bigr\}.
\]
Since the grains are closed geodesic balls of radius \(\rho\), this is
equivalent to the event that the Poisson process \(\eta\) has no point in the
\(\rho\)-tube around \(\gamma_u([0,r])\). Thus,
\[
\{s(u)>r\}
=
\bigl\{\eta\bigl(T_\rho(\gamma_u[0,r])\bigr)=0\bigr\},
\]
where
\[
T_\rho(\gamma_u[0,r])
:=
\{x\in M:d_g(x,\gamma_u([0,r]))\leq \rho\}.
\]
Therefore, by the Poisson avoidance formula,
\[
\PP\bigl(s(u)>r\bigr)
=
\exp\bigl(-\lambda\,\vol_M(T_\rho(\gamma_u[0,r]))\bigr).
\]
Similarly,
\[
\PP(o\notin Z)
=
\exp\bigl(-\lambda\,\vol_M(B(o,\rho))\bigr).
\]
Since \(B(o,\rho)\subset T_\rho(\gamma_u[0,r])\), we obtain
\begin{equation}\label{eq:Probtube}
	\PP\bigl(s(u)>r\,\big|\,o\notin Z\bigr)
	=
	\exp\Bigl(
	-\lambda\bigl[
	\vol_M(T_\rho(\gamma_u[0,r]))
	-
	\vol_M(B(o,\rho))
	\bigr]
	\Bigr).
\end{equation}
Consequently, the conditional law of \(s(u)\) is exponential whenever, for the
given radius \(\rho\), the map
\[
r\longmapsto \vol_M(T_\rho(\gamma_u[0,r]))
\]
is affine linear in \(r\) with $\vol_M(T_\rho(\gamma_u([0,0])))=\vol_M(B(o,\rho))$.

We now recall the geometric input behind this linearity. A connected
Riemannian manifold is called \emph{harmonic} if the volume density in normal
coordinates about any point depends only on the radial variable. Equivalently,
small geodesic spheres have mean curvature depending only on the radius. We
refer to the survey article of Knieper and Peyerimhoff \cite{KnieperSurvey}
for background material on (non-compact) harmonic manifolds. At the local level, harmonic manifolds are characterized by a tube property for
geodesic segments. More precisely, Csikós and Horváth
\cite[Theorem 3]{CsikosHorvathHarmonicTubes} showed that harmonicity is equivalent to the
fact that, for sufficiently small tube radii, the volume of a tube around a
geodesic segment depends only on the tube radius and on the length of the
segment.

In the homogeneous setting, the structure of harmonic manifolds is much more
rigid. By a theorem of Heber \cite[Corollary 1.2]{HeberHarmonicHomogeneous}, every simply
connected homogeneous harmonic manifold is either flat, a rank-one symmetric
space of non-compact type, or a non-symmetric Damek--Ricci space. {The
rank-one symmetric spaces of non-compact type are precisely the real, complex
and quaternionic hyperbolic spaces \(\RR H^n\), \(\CC H^m\), \(\mathbb H H^m\), together with the Cayley hyperbolic plane \(\mathbb O H^2\).  For Damek--Ricci spaces we follow \cite[Chapter 3]{DRSpaces}, use the notation
\(\mathfrak n=\mathfrak v\oplus\mathfrak z\) for the two-step nilpotent part
of the corresponding solvable Lie algebra and put
\[
p:=\dim\mathfrak v,
\qquad
q:=\dim\mathfrak z.
\]
The same notation also covers the rank-one symmetric
examples. Namely, one has $(p,q)=(n-1,0)$ for \(\RR H^n\),
and
\[
(p,q)=(2m-2,1),\qquad (p,q)=(4m-4,3),\qquad (p,q)=(8,7),
\]
for \(\CC H^m\), \(\mathbb H H^m\), and \(\mathbb O H^2\), respectively.} For background on Damek--Ricci spaces, their construction from generalized Heisenberg groups,
and their role in the theory of harmonic manifolds, we refer to \cite{DRSpaces}. We remark that these spaces do not occur in very low dimension: the first genuinely non-symmetric examples have real dimension \(7\).
Thus, up to dimension \(6\), the simply connected homogeneous harmonic manifolds are exhausted by the flat and rank-one symmetric cases. Moreover, explicit tube formulas are available in each of these families. For
Euclidean and rank-one symmetric spaces this goes back to Gray and Vanhecke
\cite{GrayVanheckeTubes}, while for Damek--Ricci spaces the corresponding
formulas were obtained by Csikós and Horváth \cite[Theorem 4.1]{CsikosHorvathTubes}.

It follows from the classification in \cite{HeberHarmonicHomogeneous} and the
explicit tube formulas from \cite{CsikosHorvathTubes,GrayVanheckeTubes} that,
for every simply connected non-compact homogeneous harmonic manifold, the
volume of a tube around a geodesic segment is an affine function of the length
of the segment. More precisely, for every fixed \(\rho>0\), there exists a
constant \(a_\rho\in(0,\infty)\), depending only on \(M\) and \(\rho\), such
that
\begin{equation}\label{eq:arho}
	\vol_M(T_\rho(\gamma_u[0,r]))
	=
	\vol_M(B(o,\rho))+a_\rho r,
	\qquad r\geq 0.
\end{equation}
The coefficient \(a_\rho\) does not depend on the base point
\(o\) or on the direction \(u\). In this formula, the term \(a_\rho r\) is the contribution of the normal tube along the geodesic segment, while the remaining term is the endpoint contribution. Since the spaces under consideration are homogeneous
and harmonic, geodesic balls have the same volume at every centre and their volume density in geodesic polar coordinates is radial. The two endpoint contributions therefore combine to \(\vol_M(B(o,\rho))\).

For example, if \(M=\RR^n\), then the tube around a line segment is the
Minkowski sum of the segment and the Euclidean ball of radius \(\rho\), and
therefore
\[
a_\rho=\kappa_{n-1}\rho^{n-1}.
\]
If \(M=\RR H^n\) is the \(n\)-dimensional real hyperbolic space with sectional
curvature \(-1\), then
\[
a_\rho=\kappa_{n-1}\sinh^{n-1}(\rho).
\]
Here and in what follows, \(\kappa_k\) denotes the volume of the
\(k\)-dimensional Euclidean unit ball.
{More generally, for rank-one symmetric spaces of non-compact type and for
Damek--Ricci spaces the coefficient can be read off from the corresponding
tube formulas. For a Damek--Ricci space with root multiplicities \(p\) and \(q\) one has
\[
a_\rho
=
\kappa_{p+q}\sinh^{p+q}(\rho)\cosh^q(\rho),
\]
see \cite[Theorem 4.1]{CsikosHorvathTubes} (it should be noted that with our normalization the distances need to be scaled by a factor \(1/2\) compared with the normalization in \cite{CsikosHorvathTubes}). 
This includes also the other rank-one symmetric spaces \(\CC H^m\), \(\HH H^m\) and
\(\OO H^2\), for which
\begin{align*}
a_\rho
&=
\kappa_{2m-1}\sinh^{2m-1}(\rho)\cosh(\rho),\qquad\,\,\text{(complex hyperbolic space)}\\
a_\rho
&=
\kappa_{4m-1}\sinh^{4m-1}(\rho)\cosh^3(\rho),\qquad\text{(quaternionic hyperbolic space)}\\
a_\rho
&=
\kappa_{15}\sinh^{15}(\rho)\cosh^7(\rho)\qquad\quad\qquad \text{(Cayley hyperbolic plane)}.
\end{align*}}

Substituting \eqref{eq:arho} into \eqref{eq:Probtube} gives
\[
\PP\bigl(s(u)>r\,\big|\,o\notin Z\bigr)
=
\exp(-\lambda a_\rho r),
\qquad r\geq 0.
\]
This proves the following theorem.

\begin{theorem}\label{thm:exp_visibility_harmonic}
	Let \(M\) be a simply connected non-compact homogeneous harmonic manifold.
	Consider the Boolean model on \(M\) with intensity \(\lambda>0\) and fixed
	ball grains of radius \(\rho>0\). Then there exists a constant
	\(a_\rho\in(0,\infty)\), depending only on \(M\) and \(\rho\), such that
	for every \(o\in M\), every unit tangent vector \(u\in T_oM\), and every
	\(r\geq 0\),
	\[
	\PP\bigl(s(u)>r\,\big|\,o\notin Z\bigr)
	=
	e^{-\lambda a_\rho r}.
	\]
\end{theorem}

\begin{remark}\label{rem:no_nonhomogeneous_known}
	At present, no non-homogeneous harmonic manifolds are known, see, for
	example, the discussion in the introduction of
	\cite{CsikosHorvathHarmonicTubes} and \cite[p.\ 100]{GrayBook}. Thus, the homogeneous setting considered
	here covers all presently known non-compact harmonic examples.
\end{remark}

\subsection{Threshold for the expected volume of the visible region}

We now turn to the volume of the visible region from \(o\). Conditioned on
\(o\notin Z\), let \(V_o\) be the set of all points \(x\in M\) such that the unique
geodesic segment from \(o\) to \(x\) does not intersect \(Z\).  More formally, for the spaces considered in this section, the usual exponential map
\(\exp_o:T_oM\to M\) is a global diffeomorphism for every \(o\in M\).
In particular, each point \(x\in M\) has a unique polar representation
\(x=\exp_o(ru)\), with \(r\geq 0\) and a unit tangent vector \(u\in T_oM\). Then we can write
\[
V_o
:=
\bigl\{\exp_o(ru): r\geq 0,\ u\in T_oM \text{ a unit vector},\ 
\gamma_u([0,r])\cap Z=\varnothing\bigr\}.
\]
Further, let
\[
h
:=
\lim_{r\to\infty}{\log\vol_M(B(o,r))\over r}
\]
be the volume entropy of \(M\). By homogeneity, this number does not depend on the choice of
\(o\). In the flat case \(M=\RR^n\), one has \(h=0\). {The volume entropy of a Damek--Ricci space with
root multiplicities \(p\) and \(q\) is
\[
h=p+2q .
\]
This convention also covers the rank-one symmetric spaces of non-compact type.
Indeed, the corresponding pairs \((p,q)\) are
\[
(n-1,0),\qquad (2m-2,1),\qquad (4m-4,3),\qquad (8,7),
\]
for \(\RR H^n\), \(\CC H^m\), \(\mathbb H H^m\), and \(\mathbb O H^2\),
respectively. Thus
\[
h=n-1,\qquad h=2m,\qquad h=4m+2,\qquad h=22
\]
in these four cases.}

The following corollary gives the threshold for the expected visible volume in terms of the tube coefficient and the volume entropy of the ambient manifold.

\begin{corollary}\label{cor}
	Let \(M\) be a simply connected non-compact homogeneous harmonic manifold
	with volume entropy \(h\geq 0\). Consider the Boolean model on \(M\) with
	intensity \(\lambda>0\) and fixed ball grains of radius \(\rho>0\), and let
	\(a_\rho\) be the tube coefficient from \eqref{eq:arho}. If \(h=0\), then $\EE[\vol_M(V_o)\mid o\notin Z]<\infty$ for every $\lambda>0$.
	If \(h>0\), then
	\[
	\EE[\vol_M(V_o)\mid o\notin Z]<\infty
	\qquad\Longleftrightarrow\qquad
	\lambda>\lambda_c(M,\rho):={h\over a_\rho}.
	\]
\end{corollary}
\begin{proof}
	By Fubini's theorem,
	\[
	\EE[\vol_M(V_o)\mid o\notin Z]
	=
	\int_M
	\PP(x\in V_o\mid o\notin Z)\,\vol_M(\dint x).
	\]
	Since the spaces
	under consideration are simply connected homogeneous harmonic manifolds, we may
	use global geodesic polar coordinates around \(o\). Then
	\begin{equation}\label{eq:intThreshold}
	\EE[\vol_M(V_o)\mid o\notin Z]
	=
	\int_0^\infty
	e^{-\lambda a_\rho r}S(r)\,\dint r,
	\end{equation}
	where \(S(r)\) denotes the surface area of the
	geodesic sphere of radius \(r\) around \(o\). Indeed, by
	Theorem~\ref{thm:exp_visibility_harmonic}, the conditional probability that
	the geodesic segment of length \(r\) from \(o\) is visible is
	\(e^{-\lambda a_\rho r}\), independently of the direction.
	
	If \(h=0\), then \(M\) is flat by Heber's classification in \cite{HeberHarmonicHomogeneous}, and hence
	\(M\) is isometric to \(\RR^n\). In this case $S(r)=n\kappa_n r^{n-1}$,
	and therefore $\int_0^\infty e^{-\lambda a_\rho r}S(r)\,\dint r<\infty$ for every \(\lambda>0\).
	
	It remains to consider the case \(h>0\). For simply connected non-compact
	homogeneous harmonic manifolds with positive entropy, the geodesic sphere
	volume has pure exponential growth with exponent \(h\), i.e., there exist
	constants \(c_1,c_2>0\) such that
	\[
	c_1e^{hr}\leq S(r)\leq c_2e^{hr},
	\qquad r\geq 1.
	\]
	This follows from the known volume density estimates for these spaces, see
	for instance the discussion before Definition~2.2 and Corollary~5.3 in
	\cite{Knieper12}. Hence, the tail behaviour of the integral in \eqref{eq:intThreshold} is the same as that of
	\[
	\int_{1}^\infty e^{-(\lambda a_\rho-h)r}\,\dint r.
	\]
	This integral is finite if and only if \(\lambda a_\rho>h\). This proves the result.
\end{proof}

Corollary~\ref{cor} recovers the critical intensity previously obtained for
the \(n\)-dimensional real hyperbolic space \(\RR H^n\) in
\cite{BuehlerHugThaeleBM}. In this case \(h=n-1\) and $a_\rho=\kappa_{n-1}\sinh^{n-1}(\rho)$. Thus,
\[
\lambda_c(\RR H^n,\rho)
=
{n-1\over \kappa_{n-1}\sinh^{n-1}(\rho)}.
\]
This agrees with \cite[Theorem~6.2 and Remark~6.5]{BuehlerHugThaeleBM} in the
case of deterministic ball grains.
{For a Damek--Ricci space with root multiplicities \(p\) and \(q\) we obtain the critical intensity
\[
\lambda_c(M,\rho)
=
{p+2q\over
	\kappa_{p+q}\sinh^{p+q}(\rho)\cosh^q(\rho)}.
\]
This formula also covers the rank-one symmetric spaces of non-compact type, for which
\[
\lambda_c(\CC H^m,\rho)
=
{2m\over
	\kappa_{2m-1}\sinh^{2m-1}(\rho)\cosh(\rho)}
\]
and
$$
\lambda_c(\HH H^m,\rho) = {4m+2\over
	\kappa_{4m-1}\sinh^{4m-1}(\rho)\cosh^3(\rho)} \qquad\text{and}\qquad \lambda_c(\OO H^2,\rho) = {22\over
	\kappa_{15}\sinh^{15}(\rho)\cosh^7(\rho)}.
$$}

\section{Examples beyond the harmonic setting}\label{sec:Example}

\subsection{A surface with Weibull-type visibility tails}

We now show that the exponential law for directional visibility is not merely a consequence of the fact that the underlying point process \(\eta\) is Poisson. It also depends on the geometry of the underlying
space. In particular, if the linear tube-growth property fails, then the
visible range need not be exponentially distributed. In what follows, we construct a complete Riemannian surface (manifold of dimension two) for which the tube volume around a
distinguished geodesic ray grows superlinearly. The corresponding visible
range will turn out to have Weibull-type tails. We refer the reader to \cite{Lee} for background material on Riemannian manifolds, particularly the local geometric concepts used in the following computations.

\begin{example}\label{ex1}
Consider \(M=\RR^2\) with global coordinates \((s,t)\) and Riemannian metric
\[
g=dt^2+J(s,t)^2\,ds^2,
\qquad
J(s,t):=1+s^2t^2 .
\]
Since \(J\) is smooth and strictly positive, this defines a smooth Riemannian
metric on \(M\). Moreover, \(J(s,t)\geq 1\) for all \((s,t)\in\RR^2\). Hence,
for every point \((s,t)\in M\) and every tangent vector $v=v_s\partial_s+v_t\partial_t\in T_{(s,t)}M$,
we have
\[
g_{(s,t)}(v,v)
=
v_t^2+J(s,t)^2v_s^2
\geq
v_s^2+v_t^2.
\]
Thus, the Riemannian length of every smooth curve is at least its Euclidean
length, and consequently the Riemannian distance \(d_g\) dominates the
Euclidean distance on \(\RR^2\). Indeed, every \(d_g\)-Cauchy sequence is a
Euclidean Cauchy sequence and hence converges in the Euclidean sense to some
point of \(\RR^2\). Since the metric \(g\) is smooth and positive definite,
the Riemannian and Euclidean distances are locally equivalent near this limit
point. So, the sequence also converges with respect to \(d_g\). Thus,
\((M,g)\) is complete. Moreover, using the definition of $J$ one can compute the Gaussian curvature $K(s,t)$ as
$$
K(s,t) = -{\partial_t^2 J(s,t)\over J(s,t)} = - {2s^2\over 1+s^2t^2}\leq 0.
$$
So, $M$ is a complete, simply connected and non-positively curved Riemannian surface, but it is non-homogeneous.

Let $o:=(0,0)$ and $\Gamma:=\{(s,0):s\geq 0\}$.
We claim that \(\Gamma\) is a geodesic ray parametrized by arc length. Indeed,
the only Christoffel symbols relevant along \(t=0\) are
\[
\Gamma^s_{ss}={\partial_s J\over J},
\qquad
\Gamma^t_{ss}=-J\,\partial_t J.
\]
Since \(J(s,0)=1\), \(\partial_sJ(s,0)=0\), and \(\partial_tJ(s,0)=0\), both
symbols vanish along \(t=0\). Therefore, the curve $\gamma(r):=(r,0)$, $r\geq 0$,
is a geodesic. Since \(J(s,0)=1\), it has unit speed.

We now estimate the volume of the \(\rho\)-tube around the segment
\(\gamma([0,r])\), where \(\rho>0\) is fixed. The Riemannian volume element of $M$ is
\[
\vol_M(\dint s\,\dint t)=J(s,t)\,\dint s\,\dint t
=
\bigl(1+s^2t^2\bigr)\,\dint s\,\dint t .
\]
For \(0\leq s\leq r\), the vertical curve from \((s,t)\) to \((s,0)\) has
length \(|t|\). Conversely, every curve from \((s,t)\) to \(\Gamma\) must
change the \(t\)-coordinate from \(t\) to \(0\), and since the metric contains
the term \(dt^2\), its length is at least \(|t|\). Hence,
\[
d_g\bigl((s,t),\Gamma([0,r])\bigr)=|t|,
\qquad 0\leq s\leq r.
\]
Consequently,
\[
[0,r]\times[-\rho,\rho]
\subset
T_\rho(\gamma[0,r]).
\]
This gives the lower bound
\[
\begin{aligned}
	\vol_M(T_\rho(\gamma[0,r]))
	&\geq
	\int_0^r\int_{-\rho}^{\rho}
	\bigl(1+s^2t^2\bigr)\,\dint t\,\dint s  =
	2\rho r+{2\rho^3\over 3}\int_0^r s^2\,\dint s  =
	2\rho r+{2\rho^3\over 9}r^3 .
\end{aligned}
\]
For the corresponding upper bound, note that the Riemannian distance dominates
the Euclidean distance. Hence,
\[
T_\rho(\gamma[0,r])
\subset
[-\rho,r+\rho]\times[-\rho,\rho],
\]
and it follows that
\[
\begin{aligned}
	\vol_M(T_\rho(\gamma[0,r]))
	&\leq
	\int_{-\rho}^{r+\rho}\int_{-\rho}^{\rho}
	\bigl(1+s^2t^2\bigr)\,\dint t\,\dint s  \\
	&=
	2\rho(r+2\rho)
	+
	{2\rho^3\over 3}
	\int_{-\rho}^{r+\rho}s^2\,\dint s  \\
	&=
	2\rho(r+2\rho)
	+
	{2\rho^3\over 9}
	\bigl((r+\rho)^3+\rho^3\bigr).
\end{aligned}
\]
The lower and upper bounds therefore imply
\begin{equation}\label{eq:superlinear_tube_asymptotic}
	\vol_M(T_\rho(\gamma[0,r]))
	\sim
	{2\rho^3\over 9}\,r^3,
	\qquad\text{as}\;\; r\to\infty,
\end{equation}
where $\sim$ indicates that the quotient of the left- and the right-hand side tends to $1$. In particular, the tube volume is not affine linear in \(r\).

Now, consider the Poisson Boolean model on \(M\) induced by a Poisson point process with intensity measure
\(\lambda\vol_M\), $\lambda>0$, and deterministic ball grains of radius \(\rho\). Let \(R\)
denote the visible range from \(o\) in the direction of the geodesic ray
\(\gamma\). As in the previous section, the Poisson avoidance formula gives
\[
\PP(R>r\mid o\notin Z)
=
\exp\Bigl(
-\lambda\bigl[
\vol_M(T_\rho(\gamma[0,r]))
-
\vol_M(B(o,\rho))
\bigr]
\Bigr).
\]
So, using \eqref{eq:superlinear_tube_asymptotic}, we obtain
\[
\log \PP(R>r\mid o\notin Z)
\sim
-\lambda\,{2\rho^3\over 9}\,r^3,
\qquad\text{as}\;\; r\to\infty .
\]
Thus, the conditional visible range in the direction of $\gamma$ has Weibull-type tails with exponent
\(3\). 
\end{example}

\subsection{Examples with asymptotic tube linearity}

The construction in the preceding subsection shows that, without the harmonic or homogeneous structure, the visible range need not be exponentially distributed. There is, however, an intermediate situation which is worth recording. Even if the tube volume is not exactly affine linear in \(r\), it may still be asymptotically linear. In that case the visible range need not be exponentially distributed, but its tail has still an exponential decay rate. Indeed, if for a geodesic ray \(\gamma\) the limit
\[
a_\rho(\gamma)
:=
\lim_{r\to\infty}
{1\over r}\vol_M(T_\rho(\gamma[0,r]))
\]
exists and belongs to \((0,\infty)\), then the Poisson avoidance formula gives
\[
-\frac{1}{r}\log \PP(R>r\mid o\notin Z)
\to
\lambda a_\rho(\gamma),
\qquad\text{as}\;\; r\to\infty,
\]
where \(R\) denotes the visible range from \(o\) in the direction of \(\gamma\). Indeed,
\[
\PP(R>r\mid o\notin Z)
=
\exp\Bigl(
-\lambda\bigl[
\vol_M(T_\rho(\gamma[0,r]))
-
\vol_M(B(o,\rho))
\bigr]
\Bigr),
\]
and the term \(\vol_M(B(o,\rho))\) is independent of \(r\). Thus, exact tube
linearity yields an exact exponential distribution, while asymptotic tube
linearity yields an exponential decay rate. We present two  examples in which the latter phenomenon occurs.

\begin{example}\label{ex2}
	Let \(M=\RR^2\) with global coordinates \((s,t)\), and equip \(M\) with the
	Riemannian metric
	\[
	g=dt^2+J(s,t)^2\,ds^2,
	\qquad
	J(s,t):=1+\varepsilon q(s)t^2,
	\]
	where \(\varepsilon>0\) and \(q:\RR\to(0,\infty)\) is smooth, bounded and
	non-constant. Assume moreover that the Cesàro mean
	\[
	\bar q
	:=
	\lim_{r\to\infty}{1\over r}\int_0^r q(s)\,\dint s
	\]
	exists. As a concrete example one may take $q(s)=1+{1\over 2}\sin s$ for which $\bar q=1$. Since \(J\geq1\), the argument from Example~\ref{ex1} shows that
	\((M,g)\) is complete.
	
	Let
	\[
	o:=(0,0),
	\qquad
	\Gamma:=\{(s,0):s\geq0\},
	\qquad\text{and}\qquad
	\gamma(r):=(r,0).
	\]
	Since \(J(s,0)=1\), \(\partial_sJ(s,0)=0\), and
	\(\partial_tJ(s,0)=0\), the curve \(\gamma\) is a unit-speed geodesic by the
	same Christoffel-symbol computation as in Example~\ref{ex1}.
	
	The surface $M$ is in general neither homogeneous nor harmonic. Indeed, for a
	metric of the form $g=dt^2+J(s,t)^2\,ds^2$
	the Gaussian curvature is
	\[
	K(s,t)=-{\partial_t^2J(s,t)\over J(s,t)}.
	\]
	In the present case this gives $K(s,0)=-2\varepsilon q(s)$,
	which is not constant if \(q\) is not constant. Hence, the surface is not
	homogeneous. Since a harmonic surface has constant Gaussian curvature, it is
	not harmonic either.
	
	For \(0\leq s\leq r\), the vertical curve from \((s,t)\) to \((s,0)\) has
	length \(|t|\), and every curve from \((s,t)\) to \(\Gamma([0,r])\) has length
	at least \(|t|\). Hence,
	\[
	d_g\bigl((s,t),\Gamma([0,r])\bigr)=|t|,
	\qquad 0\leq s\leq r.
	\]
	Since \(J\geq1\), the Riemannian distance dominates the Euclidean distance.
	Therefore,
	\[
	[0,r]\times[-\rho,\rho]
	\subset
	T_\rho(\gamma[0,r])
	\subset
	[-\rho,r+\rho]\times[-\rho,\rho].
	\]
	The contribution of the two endpoint strips is bounded uniformly in \(r\),
	because \(q\) is bounded. Consequently,
	\[
	\vol_M(T_\rho(\gamma[0,r]))
	=
	\int_0^r\int_{-\rho}^{\rho}
	\bigl(1+\varepsilon q(s)t^2\bigr)\,\dint t\,\dint s
	+
	O(1).
	\]
	It follows that
	\[
	\vol_M(T_\rho(\gamma[0,r]))
	=
	2\rho r
	+
	{2\varepsilon\rho^3\over3}
	\int_0^r q(s)\,\dint s
	+
	O(1),
	\]
	and hence
	\[
	{1\over r}\vol_M(T_\rho(\gamma[0,r]))
	\to
	2\rho+{2\varepsilon\rho^3\over3}\bar q \qquad\text{as}\;\; r\to\infty.
	\]
	Thus, for the Poisson Boolean model with intensity measure
	\(\lambda\vol_M\) and deterministic ball grains of radius $\rho$, the visible range in the direction of \(\gamma\) satisfies
	\[
	-\frac{1}{r}\log \PP(R>r\mid o\notin Z)
	\to
	\lambda\left(2\rho+{2\varepsilon\rho^3\over3}\bar q\right),\qquad\text{as}\;\; r\to\infty.
	\]
\end{example}

\begin{example}\label{ex3}
	The construction of Example~\ref{ex2} has a higher-dimensional analogue. Let
	\(n\geq2\), let \(M=\RR\times\RR^{n-1}\) with global coordinates \((s,y)\),
	where \(y\in\RR^{n-1}\), and equip \(M\) with the Riemannian metric
	\[
	g=\|dy\|^2+J(s,y)^2\,ds^2,
	\qquad
	J(s,y):=1+\varepsilon q(s)\phi(y)
	\]
	 with the Euclidean norm $\|\cdot\|$ in $\RR^{n-1}$.
	Assume that \(q:\RR\to(0,\infty)\) is smooth, bounded and non-constant, and
	that its Cesàro mean
	\[
	\bar q
	:=
	\lim_{r\to\infty}{1\over r}\int_0^r q(s)\,\dint s
	\]
	exists. 
	Assume also that \(\phi:\RR^{n-1}\to[0,\infty)\) is smooth and bounded
	on bounded sets, with $\phi(0)=0$ and  $\nabla\phi(0)=0$. For concreteness one may take $q(s)=1+{1\over 2}\sin s$ with $\bar q=1$ and $\phi(y)=\|y\|^2$.
	Then \(J\geq1\), and hence \(g\) is a smooth complete Riemannian metric on
	\(\RR^n\).
	
	As before, let
	\[
	o:=(0,0),
	\qquad
	\Gamma:=\{(s,0):s\geq0\},
	\qquad\text{and}\qquad
	\gamma(r):=(r,0).
	\]
	Since \(J(s,0)=1\), \(\partial_sJ(s,0)=0\), and
	\(\nabla_yJ(s,0)=0\), the relevant Christoffel symbols vanish along
	\(\Gamma\). Hence, \(\gamma\) is a unit-speed geodesic.
	
	The transverse distance to \(\Gamma([0,r])\) is exactly Euclidean for
	\(0\leq s\leq r\), because the restriction of the metric to each slice
	\(\{s=\mathrm{const}\}\) is \(\|dy\|^2\). Thus,
	\[
	d_g\bigl((s,y),\Gamma([0,r])\bigr)=\|y\|,
	\qquad 0\leq s\leq r.
	\]
	Moreover, since \(J\geq1\), the Riemannian distance dominates the Euclidean
	distance, and therefore
	\[
	[0,r]\times B_{n-1}(0,\rho)
	\subset
	T_\rho(\gamma[0,r])
	\subset
	[-\rho,r+\rho]\times B_{n-1}(0,\rho),
	\]
	where $B_{n-1}(0,\rho)$ is the centred ball of radius $\rho$ in the $\RR^{n-1}$-coordinate.
	Since \(q\) is bounded and \(\phi\) is bounded on \(B_{n-1}(0,\rho)\), the
	endpoint contribution is \(O(1)\). Hence,
	\[
	\vol_M(T_\rho(\gamma[0,r]))
	=
	\int_0^r\int_{B_{n-1}(0,\rho)}
	\bigl(1+\varepsilon q(s)\phi(y)\bigr)\,\dint y\,\dint s
	+
	O(1).
	\]
	Consequently,
	\[
	\vol_M(T_\rho(\gamma[0,r]))
	=
	\kappa_{n-1}\rho^{n-1}r
	+
	\varepsilon
	\left(\int_{B_{n-1}(0,\rho)}\phi(y)\,\dint y\right)
	\int_0^r q(s)\,\dint s
	+
	O(1),
	\]
	and therefore
	\[
	{1\over r}\vol_M(T_\rho(\gamma[0,r]))
	\to
	\kappa_{n-1}\rho^{n-1}
	+
	\varepsilon\bar q
	\int_{B_{n-1}(0,\rho)}\phi(y)\,\dint y \qquad\text{as}\;\; r\to\infty.
	\]
	Thus, for the Poisson Boolean model with intensity measure
	\(\lambda\vol_M\) and deterministic ball grains of radius $\rho$, the visible range in the direction of \(\gamma\) satisfies
	\[
	-\frac{1}{r}\log \PP(R>r\mid o\notin Z)
	\to
	\lambda\left(
	\kappa_{n-1}\rho^{n-1}
	+
	\varepsilon\bar q
	\int_{B_{n-1}(0,\rho)}\phi(y)\,\dint y
	\right),\qquad\text{as}\;\; r\to\infty.
	\]
	If $q(s)=1+{1\over 2}\sin s$ and $\phi(y)=\|y\|^2$ the second term is $\varepsilon\bar q\int_{B_{n-1}}\phi(y)\,\dint y=\varepsilon{(n-1)\kappa_{n-1}\over n+1}\rho^{n+1}$. 
\end{example}

In particular, Examples \ref{ex2} and \ref{ex3}  show that exact homogeneity is not necessary for an exponential
tail rate. What is needed for this conclusion is the existence of an averaged
transverse tube density along the chosen geodesic.

\subsection*{Acknowledgment}
This work was initiated during a visit of the second author to the University of Lausanne, whose kind hospitality is gratefully acknowledged. The second author would also like to thank Gerhard Knieper for introducing him some years ago to the world of harmonic manifolds.

The authors used ChatGPT as a writing and discussion tool in preparing the manuscript. The mathe\-matical content and any remaining errors are the responsibility of the authors.


\phantomsection

\end{document}